\documentclass[12pt]{article}
\usepackage{amsmath,latexsym,amssymb}
\usepackage{amsthm}
\usepackage{epsfig}

\newcommand\NN{\mathbb{N}}

\newcommand\ZZ{\mathbb{Z}}

\newcommand\BB{\mathcal{B}}

\newcommand\GG{\mathcal{G}}
\newcommand\JJ{\mathcal{J}}

\newcommand\eps{{\varepsilon}}

\DeclareMathOperator\dist{dist}

\DeclareMathOperator\diam{diam}

\newtheorem{theorem}{Theorem}[section]
\newtheorem{definition}[theorem]{Definition}
\newtheorem{corollary}[theorem]{Corollary}
\newtheorem{example}[theorem]{Example}
\newtheorem{lemma}[theorem]{Lemma}
\newtheorem{remark}[theorem]{Remark}
\newtheorem{proposition}[theorem]{Proposition}
\newtheorem{conjecture}[theorem]{Conjecture}

\newcommand{\address}{Address: Department of Mathematics, University of North Texas, 1155 Union Circle \#311430, Denton, TX 76203-5017, USA; E-mail: allaart@unt.edu}

\numberwithin{equation}{section}

\title{On the distribution of the cardinalities of level sets of the Takagi function}
\author{Pieter C. Allaart \footnote{\address}}

\begin{document}

\maketitle

\begin{abstract}

Let $T$ be Takagi's continuous but nowhere-differentiable function. It is known that almost all level sets (with respect to Lebesgue measure on the range of $T$) are finite. We show that the most common cardinality of the level sets of $T$ is two, and investigate in detail the set of ordinates $y$ such that the level set at level $y$ has precisely two elements. As a by-product, we obtain a simple iterative procedure for solving the equation $T(x)=y$. We show further that any positive even integer occurs as the cardinality of some level set, and investigate which cardinalities occur with positive probability if an ordinate $y$ is chosen at random from the range of $T$. The key to the results is a system of set equations for the level sets, which are derived from the partial self-similarity of $T$. These set equations yield a system of linear relationships between the cardinalities of level sets at various levels, from which all the results of this paper flow.


\end{abstract}

\section{Introduction}

Let $\phi(x)=\dist(x,\ZZ)$ be the distance from the point $x$ to the nearest integer. Takagi's function is defined by
\begin{equation}
T(x)=\sum_{n=0}^\infty \frac{1}{2^n}\phi(2^n x).
\label{eq:Takagi-def}
\end{equation}
It is continuous but nowhere differentiable (Takagi \cite{Takagi}, Hildebrandt \cite{Hildebrandt}, de Rham \cite{deRham}, Billingsley \cite{Billingsley}), though it does possess an infinite derivative at many points (see Allaart and Kawamura \cite{AK} or Kr\"uppel \cite{Kruppel} for a precise characterization). 

In recent years, there has been a great deal of interest in the level sets 
$$L(y):=\{x\in[0,1]: T(x)=y\}, \qquad y\geq 0$$
of the Takagi function restricted to the unit interval. Kahane \cite{Kahane} had already shown in 1959 that the maximum value of $T$ is $\frac23$, and that $L(\frac23)$ is the set of all $x\in[0,1]$ whose binary expansion $x=0.b_1 b_2 b_3\dots$ satisfies $b_{2i-1}+b_i=1$ for all $i\in\NN$. This is equivalent to saying that the quarternary expansion of $x$ contains only $1$'s and $2$'s. As a result, $L(\frac23)$ is a Cantor set of Hausdorff dimension $\frac12$. Surprisingly, however, a more general study of the level sets of $T$ was not undertaken until fairly recently, when Knuth \cite[p.~103]{Knuth} published an algorithm for generating solutions of the equation $T(x)=y$ for rational $y$. (It is however not known whether his algorithm always halts.) Buczolich \cite{Buczolich} showed, among other things, that almost all level sets of $T$ (with respect to Lebesgue measure) are finite. Shortly afterwards, Maddock \cite{Maddock} proved that the Hausdorff dimension of any level set is at most $0.668$, and conjectured an upper bound of $\frac12$. This upper bound was very recently verified by de Amo et al. \cite{deAmo}. Lagarias and Maddock \cite{LagMad1,LagMad2} introduced the concept of a {\em local level set} to prove a number of new results. For instance, they showed that $L(y)$ is countably infinite for a dense set of $y$-values; that the {\em average} cardinality of all level sets if infinite; and that the set of ordinates $y$ for which $L(y)$ has strictly positive Hausdorff dimension is of full Hausdorff dimension 1. Combined with the result of Buczolich, these results sketch a complex picture of the totality of level sets of the Takagi function.

In a related paper by the present author \cite{Allaart}, shorter proofs are given for some of the above-mentioned results and the level sets are examined from the point of view of Baire category. One of the main results is that the typical level set of $T$ is uncountably large, thus providing a further contrast with Buczolich's theorem. 

While thus far, most of the research has focused on the sizes of the infinite level sets (both in terms of cardinality and Hausdorff dimension), the present article aims to give new insight in the cardinalities of the {\em finite} level sets of $T$. We investigate in detail the set of ordinates $y$ for which $L(y)$ has precisely two elements, and show that this is the most common possibility. We then show that every positive even number occurs as the size of some level set of $T$, and examine which even numbers occur with positive probability if an ordinate $y$ is chosen at random from $[0,\frac23]$.

The paper is organized as follows. Section \ref{sec:prelim} introduces notation and recalls some important facts about the Takagi function, including its functional equation and partial self-similarity. Section \ref{sec:set-equation} shows that the level sets satisfy a system of set equations, which are fundamental to the results in this paper. From the set equations, we immediately obtain simple linear relationships between the cardinalities of level sets at various levels.

Section \ref{sec:two-elements} deals with those level sets having exactly two elements. From the fundamental set equations of Section \ref{sec:set-equation} we quickly obtain a precise characterization of the set $S_2$ of ordinates $y$ such that $|L(y)|=2$. However, the condition is in general difficult to check, so we give several easier to verify conditions which are either sufficient or necessary for membership in $S_2$. Next, we show that the Lebesgue measure of $S_2$ is between $5/12$ and $35/72$. From a probabilistic point of view, this means that if an ordinate $y$ is chosen at random in the range $[0,\frac23]$, the probability that $|L(y)|=2$ is more than $5/8$, or $62.5\%$, but less than $35/48$, or $72.9\%$. In particular, $2$ is the most common cardinality of level sets of $T$ in the sense of measure. To end the section, we introduce the {\em Takagi expansion} of a point $y$ in $[0,\frac23]$, based on the notation developed in Section \ref{sec:two-elements}, and use it to give a simple iterative procedure for solving the equation $T(x)=y$. We also point out a direct connection between Takagi expansions and the local level sets of Lagarias and Maddock \cite{LagMad1}.

In Section \ref{sec:even-cardinalities} we show that every positive even number occurs as the cardinality of some (in fact, uncountably many) level sets of $T$. We conjecture that the Lebesgue measure of the set $S_{2n}:=\{y\in[0,\frac23]: |L(y)|=2n\}$ is positive for each $n$, but are able to prove this only for the case when $n$ is either a power of $2$, or the sum or difference of two powers of $2$.

\section{Preliminaries} \label{sec:prelim}

In this paper, $|.|$ will always denote cardinality; the diameter of a set $A$ will be denoted by $\diam(A)$.

We first recall some known facts about the Takagi function, and introduce important notation. One of the foremost tools for analyzing the Takagi function is its functional equation -- see, for instance, Hata and Yamaguti \cite{Hata-Yamaguti} or Kairies et al. \cite{KaiDarFra}.

\begin{lemma}[The functional equation] \label{lem:FE}

(i) The Takagi function is symmetric about $x=\frac12$:
\begin{equation}
T(1-x)=T(x) \qquad\mbox{for all $x\in[0,1]$}.
\label{eq:symmetry}
\end{equation}

(ii) The Takagi function satisfies the functional equation
\begin{equation}
T(x)=\begin{cases}
\frac12 T(2x)+x, & \mbox{if $0\leq x\leq \frac12$},\\
\frac12 T(2x-1)+1-x, & \mbox{if $\frac12\leq x\leq 1$}.
\end{cases}
\label{eq:FE}
\end{equation}
\end{lemma}

It is well known, but not needed here, that $T$ is also the unique bounded solution of \eqref{eq:FE}.

Next, we define the {\em partial Takagi functions}
\begin{equation}
T_k(x):=\sum_{n=0}^{k-1}\frac{1}{2^n}\phi(2^n x), \qquad k=1,2,\dots.
\end{equation}
Each function $T_k$ is piecewise linear with integer slopes. In fact, the slope of $T_k$ at a non-dyadic point $x$ is easily expressed in terms of the binary expansion of $x$. We define the binary expansion of $x\in[0,1)$ by
\begin{equation*}
x=\sum_{n=1}^\infty \frac{\eps_n}{2^n}=0.\eps_1\eps_2\dots\eps_n\dots, \qquad\eps_n\in\{0,1\},
\end{equation*}
with the convention that if $x$ is dyadic rational, we choose the representation ending in all zeros. For $k=0,1,2,\dots$, let
\begin{equation*}
D_k(x):=\sum_{j=1}^k(1-2\eps_j)=\sum_{j=1}^k(-1)^{\eps_j}
\end{equation*}
denote the excess of $0$ digits over $1$ digits in the first $k$ binary digits of $x$. Then it follows directly from  
\eqref{eq:Takagi-def} that the slope of $T_k$ at a non-dyadic point $x$ is $D_k(x)$.

The first part of the following definition is taken from Lagarias and Maddock \cite{LagMad1}.

\begin{definition} \label{def:balanced}
{\rm
A dyadic rational of the form $x=0.\eps_1\eps_2\dots\eps_{2m}$ is called {\em balanced} if $D_{2m}(x)=0$. If there are exactly $n$ indices $1\leq j\leq 2m$ such that $D_j(x)=0$, we say $x$ is a balanced dyadic rational of {\em generation} $n$. By convention, we consider $x=0$ to be a balanced dyadic rational of generation $0$.
}
\end{definition}

The next lemma states in a precise way that the graph of $T$ contains everywhere small-scale similar copies of itself. Let
\begin{equation*}
\GG_T:=\{(x,T(x)): 0\leq x\leq 1\}
\end{equation*}
denote the graph of $T$ over the unit interval $[0,1]$.

\begin{lemma}[Self-similarity] \label{lem:similar-copies}
Let $m\in\NN$, and let $x_0=k/2^{2m}=0.\eps_1\eps_2\dots\eps_{2m}$ be a balanced dyadic rational. Then for $x\in[k/2^{2m},(k+1)/2^{2m}]$ we have
\begin{equation*}
T(x)=T(x_0)+\frac{1}{2^{2m}}T\left(2^{2m}(x-x_0)\right).
\end{equation*}
In other words, the part of the graph of $T$ above the interval $[k/2^{2m},(k+1)/2^{2m}]$ is a similar copy of the full graph $\GG_T$, reduced by a factor $1/2^{2m}$ and shifted up by $T(x_0)$.
\end{lemma}

\begin{proof}
This follows immediately from the definition \eqref{eq:Takagi-def}, since the slope of $T_{2m}$ over the interval $[k/2^{2m},(k+1)/2^{2m}]$ is equal to $D_{2m}(x_0)=0$, and $T(x_0)=T_{2m}(x_0)$.
\end{proof}

\begin{definition} \label{def:humps}
{\rm
For a balanced dyadic rational $x_0=k/2^{2m}$ as in Lemma \ref{lem:similar-copies}, define 
\begin{gather*}
I(x_0)=[k/2^{2m},(k+1)/2^{2m}], \qquad J(x_0)=T(I(x_0)),\\
K(x_0)=I(x_0)\times J(x_0),\\
H(x_0)=\mathcal{G}_T\cap K(x_0).
\end{gather*}
By Lemma \ref{lem:similar-copies}, $H(x_0)$ is a similar copy of the full graph $\mathcal{G}_T$; we call it a {\em hump}. Its height is $\diam(J(x_0))=\frac23{(\frac14)}^m$, and we call $m$ its {\em order}. By the {\em generation} of the hump $H(x_0)$ we mean the generation of the balanced dyadic rational $x_0$. A hump of generation $1$ will be called a {\em first-generation hump}. By convention, the graph $\mathcal{G}_T$ itself is a hump of generation $0$. If $D_j(x_0)\geq 0$ for every $j\leq 2m$, we call $H(x_0)$ a {\em leading hump}. See Figure \ref{fig:humps} for an illustration of these concepts.
}
\end{definition}

\begin{figure}
\begin{center}
\epsfig{file=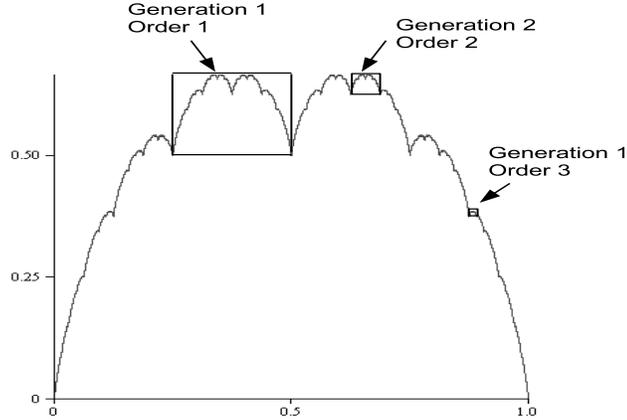, height=.35\textheight, width=.7\textwidth}
\caption{The graph of $T$, with humps of various orders and generations highlighted. The rectangles shown are, from left to right, $K(1/4)$, $K(5/8)$ and $K(7/8)$. Note that in binary, $1/4=0.01$, $5/8=0.1010$, and $7/8=0.111000$.}
\label{fig:humps}
\end{center}
\end{figure}

\section{The fundamental set equations} \label{sec:set-equation}

In this section we establish a system of set equations for the level sets $L(y)$. These equations, which describe the complex relationships between the level sets at various levels, hold the key to later results. We shall need the following notation.

First, define for $k\in\NN$ the affine maps 
\begin{equation*}
f_k(x)=\frac{x}{4}+\frac{1}{2^k}
\end{equation*}
and
\begin{equation*}
g_{k,j}(x)=\frac{x}{4^{k+j}}+\frac{1}{2^k}-\sum_{r=0}^j \frac{1}{4^{k+r}}, \qquad j=0,1,2,\dots.
\end{equation*}

Next, for ease of notation, let
\begin{equation*}
t_k:=\frac{k}{2^k}, \qquad k\in\NN.
\end{equation*}
Observe that $t_1=t_2=\frac12$, and thereafter $t_k$ is strictly decreasing in $k$. Define intervals
\begin{equation*}
I_2:=[\tfrac12,\tfrac23], \qquad\mbox{and}\qquad I_k:=[t_k,t_{k-1}), \quad k\geq 3,
\end{equation*}
and note that $[0,\frac23]=\{0\}\cup \bigcup_{k=2}^\infty I_k$. Define mappings $\Psi_k:[0,\frac23]\to [0,\infty)$ by
\begin{equation*}
\Psi_k(y)=4(y-t_k)\chi_{I_k}(y)
\end{equation*}
for $k=2,3,\dots$, where $\chi_A$ denotes the characteristic function of the set $A$. Let $\Psi:=\sum_{k=2}^\infty \Psi_k$, and let $\Psi^n$ denote the $n$-fold iteration of $\Psi$, with $\Psi^0(y):=y$. Note that $\Psi$ maps $I_2$ onto $[0,\frac23]$. For $k\geq 3$, we have
\begin{equation*}
4(t_{k-1}-t_k)=4\left(\frac{k-1}{2^{k-1}}-\frac{k}{2^k}\right)=\frac{k-2}{2^{k-2}}=t_{k-2},
\end{equation*}
so $\Psi$ maps $I_k$ onto $[0,t_{k-2})=\{0\}\cup\bigcup_{j=k-1}^\infty I_j$. As a result, $\Psi$ is a surjective self-map of $[0,\frac23]$ which maps $[0,\frac12)$ onto itself.

Finally, define a map $\Phi:[0,\infty)\to[0,\infty)$ by 
\begin{equation*}
\Phi(y):=\begin{cases}
0, & \mbox{if $y=0$},\\
4^k(y-t_k), & \mbox{if $y\in I_k$} \quad (k=3,4,\dots),\\
4(y-\frac12), & \mbox{if $y\geq\frac12$}.
\end{cases}
\end{equation*}

\begin{theorem} \label{thm:fundamental-equation}
For $y\geq 0$, let 
$$L_0(y)=\{x\in[0,\tfrac12]: T(x)=y\}=L(y)\cap[0,\tfrac12].$$
 
(i) For all $y$,
\begin{equation}
L(y)=L_0(y)\cup [1-L_0(y)],
\label{eq:set-equation1}
\end{equation} 
where the union is disjoint except when $y=\frac12$.

(ii) For $y\in I_k$ ($k\geq 3$), we have
\begin{equation}
L_0(y)=f_k\left[L_0(\Psi(y))\right]\cup\bigcup_{j=0}^\infty g_{k,j}\left[L(4^j\Phi(y))\right],
\label{eq:set-equation2}
\end{equation}
where the union is completely disjoint except when $y=t_k$.

(iii) For $y\in I_2$,
\begin{equation}
L_0(y)=\bigcup_{j=0}^\infty g_{1,j}\left[L(4^j(\Phi(y))\right],
\label{eq:set-equation3}
\end{equation}
with the union disjoint except when $y=\frac12$.
\end{theorem}

Theorem \ref{thm:fundamental-equation}, which is proved at the end of this section, has the following immediate consequence for the cardinalities of the level sets of $T$.

\begin{corollary} \label{cor:cardinality}
(i) For each $y$, 
\begin{equation}
|L(y)|=2|L_0(y)|.
\label{eq:card-equation1}
\end{equation}

(ii) If $y\in I_k$ for $k\geq 3$, then
\begin{equation}
|L_0(y)|=|L_0(\Psi(y))|+\sum_{j=0}^\infty|L(4^j\Phi(y))|. 
\label{eq:card-equation2}
\end{equation}

(iii) If $y\in I_2$, then
\begin{equation}
|L_0(y)|=\sum_{j=0}^\infty |L(4^j\Phi(y))|.
\label{eq:card-equation3}
\end{equation}
\end{corollary}

\begin{proof}
It is well known (e.g. \cite[Theorem 6.1]{LagMad1}) that $L(\frac12)$ is countably infinite. Thus, \eqref{eq:card-equation1} and \eqref{eq:card-equation3} take the form $\infty=\infty$ for $y=\frac12$. Similarly, if $y=t_k$ for $k\geq 3$, then it follows easily from \eqref{eq:set-equation2} that both sides of \eqref{eq:card-equation2} are infinite. For all other values of $y$, the equalities are obvious from the disjointness mentioned in Theorem \ref{thm:fundamental-equation}.
\end{proof}

To prove Theorem \ref{thm:fundamental-equation}, we first need the following extension of Lemma \ref{lem:similar-copies}.

\begin{lemma} \label{lem:series-of-humps}
Let $x=0.\eps_1\eps_2\dots\eps_{2m}$ be a balanced dyadic rational of order $m$ such that $\eps_{2m}=1$. Define
\begin{equation*}
x_j:=x-\sum_{r=1}^j\frac{1}{4^{m+r}}, \qquad j=0,1,2,\dots.
\end{equation*}
Then, for each $j\in\NN$, the graph of $T$ above the interval $[x_j,x_{j-1}]$ is a similar copy of $\GG_T$, scaled by a factor $1/4^{m+j}$ and shifted vertically by $T(x)$. More precisely,
\begin{equation*}
T(\xi)=T(x)+\frac{1}{4^{m+j}}T\left(4^{m+j}(\xi-x_j)\right), \qquad \xi\in[x_j,x_{j-1}].
\end{equation*}
\end{lemma}

\begin{proof}
The binary expansion of $x_j$ is $x_j=0.\eps_1\eps_2\dots\eps_{2m-1}(01)^j 1$. Since $x$ is balanced of order $m$, $D_{2m}(x)=0$, and hence $D_{2m+2j}(x_j)=0$. Thus, $x_j$ is a balanced dyadic rational of order $m+j$, and the statement of the lemma follows by Lemma \ref{lem:similar-copies}.
\end{proof}

The next lemma is a self-similarity result.

\begin{lemma} \label{lem:self-similar}
Let $k\geq 2$. If 
\begin{equation}
\frac{1}{2^k}\leq x\leq \frac{1}{2^{k-1}},
\label{eq:x-interval}
\end{equation}
then
\begin{equation}
T(x)=\frac{k}{2^k}+\frac14 T\left(4\left(x-\frac{1}{2^k}\right)\right).
\label{eq:local-similarity}
\end{equation}
\end{lemma}

\begin{proof}
If $x$ satisfies \eqref{eq:x-interval}, then iterating the first half of the functional equation \eqref{eq:FE} $k-1$ times yields
\begin{equation}
T(x)=\frac{1}{2^{k-1}}T(2^{k-1}x)+(k-1)x.
\label{eq:FE-iteration1}
\end{equation}
Similarly,
\begin{equation}
T\left(4\left(x-\frac{1}{2^k}\right)\right)=\frac{1}{2^{k-2}}T\left(2^k\left(x-\frac{1}{2^k}\right)\right)+4(k-2)\left(x-\frac{1}{2^k}\right).
\label{eq:FE-iteration2}
\end{equation}
Now by the second half of \eqref{eq:FE},
\begin{equation*}
T(2^{k-1}x)=\frac12 T(2^k x-1)+1-2^{k-1}x.
\end{equation*}
Substituting this into \eqref{eq:FE-iteration1} yields
\begin{equation}
T(x)=\frac{1}{2^k}T(2^k x-1)+\frac{1}{2^{k-1}}+(k-2)x.
\label{eq:Tx-final}
\end{equation}
From \eqref{eq:FE-iteration2} and \eqref{eq:Tx-final}, \eqref{eq:local-similarity} follows easily.
\end{proof}

\begin{lemma} \label{lem:pseudo-monotone}
Let $k\geq 2$. If ${(\frac12)}^k\leq x\leq \frac12$, then
\begin{equation*}
T(x)\geq T\left(\frac{1}{2^k}\right)=\frac{k}{2^k}.
\end{equation*}
\end{lemma}

\begin{proof}
Letting $x=1/2^k$ in \eqref{eq:local-similarity} we obtain $T(1/2^k)=k/2^k$. If ${(\frac12)}^k\leq x\leq \frac12$, then we can find an integer $l$ with $2\leq l\leq k$ such that ${(\frac12)}^l\leq x\leq {(\frac12)}^{l-1}$.
Since the slope of $T_l$ over the interval $\big[{(\frac12)}^l,{(\frac12)}^{l-1}\big]$ is $D_l(x)=(l-1)-1=l-2\geq 0$, we can conclude that
\begin{equation*}
T(x)\geq T_l(x)\geq T_l\left(\frac{1}{2^l}\right)=T\left(\frac{1}{2^l}\right)=\frac{l}{2^l}\geq\frac{k}{2^k}=T\left(\frac{1}{2^k}\right),
\end{equation*}
where the last inequality follows since $k/2^k$ is nonincreasing.
\end{proof}

\begin{lemma} \label{lem:left-side}
Let $k\geq 2$, and for $j=-1,0,1,\dots$, put
\begin{equation*}
x_{k,j}:=\frac{1}{2^k}-\sum_{r=0}^{j}\frac{1}{4^{k+r}},
\end{equation*}
where the empty sum is taken to be zero.

(i) If $j\geq 0$ and $x_{k,j}\leq x\leq x_{k,j-1}$, then
\begin{equation}
T(x)=T\left(\frac{1}{2^k}\right)+\frac{1}{4^{k+j}}T\left(4^{k+j}(x-x_{k,j})\right).
\label{eq:small-humps-to-the-left}
\end{equation}
In other words, the portion of the graph of $T$ above the interval $[x_{k,j},x_{k,j-1}]$ is a similar copy of the whole graph of $T$, scaled by $1/4^{k+j}$ and positioned with its base at the level $y=T(1/2^k)$.

(ii) If
\begin{equation*}
x<\lim_{j\to\infty}x_{k,j}=\frac{1}{2^k}-\frac{1}{3\cdot 4^{k-1}},
\end{equation*}
then $T(x)<T(1/2^k)$.
\end{lemma}

\begin{proof}
(i) Note that $x_{k,0}=0.0^k 1^k$, so $x_{k,0}$ satisfies the hypothesis of Lemma \ref{lem:series-of-humps} with $m=k$. Thus, \eqref{eq:small-humps-to-the-left} is a consequence of Lemmas \ref{lem:similar-copies} and \ref{lem:series-of-humps}.

(ii) We prove the second statement by induction. First, if $x<\lim_{j\to\infty}x_{2,j}=\frac16$, then
\begin{equation}
T(x)=\frac12 T(2x)+x<\frac12\cdot\frac23+\frac16=\frac12=T\left(\frac14\right).
\label{eq:basis-case}
\end{equation}
(This was observed also by Lagarias and Maddock \cite[Section 6]{LagMad1}.)

For the induction step, we first show that for $k\geq 2$,
\begin{equation}
\mbox{if}\quad x<\frac{1}{2^{k+1}}, \quad\mbox{then}\quad T(x)<\frac{k}{2^k}.
\label{eq:monotone-like}
\end{equation}
This holds for $k=2$ in view of \eqref{eq:basis-case}. Suppose \eqref{eq:monotone-like} holds for some arbitrary $k\geq 2$, and let $x<1/2^{k+2}$; then
\begin{equation*}
T(x)=\frac12 T(2x)+x<\frac12\cdot \frac{k}{2^k}+\frac{1}{2^{k+2}}=\frac{2k+1}{2^{k+2}}<\frac{k+1}{2^{k+1}}.
\end{equation*}
Thus, by induction, \eqref{eq:monotone-like} holds for every $k\geq 2$.

Suppose now that statement (ii) is true for some arbitrary $k\geq 2$. If $x<1/2^{k+1}$, then $T(x)<k/2^k=T(1/2^k)$ by \eqref{eq:monotone-like}. On the other hand, if
\begin{equation*}
\frac{1}{2^{k+1}}\leq x<\frac{1}{2^k}-\frac{1}{3\cdot 4^{k-1}},
\end{equation*}
we can apply Lemma \ref{lem:self-similar}: Since
\begin{equation*}
4\left(x-\frac{1}{2^{k+1}}\right)<\frac{1}{2^{k-2}}-\frac{1}{3\cdot 4^{k-2}}-\frac{1}{2^{k-1}}=\frac{1}{2^{k-1}}-\frac{1}{3\cdot 4^{k-2}},
\end{equation*}
the induction hypothesis gives
\begin{equation*}
T\left(4\left(x-\frac{1}{2^{k+1}}\right)\right)<T\left(\frac{1}{2^{k-1}}\right)=\frac{k-1}{2^{k-1}}.
\end{equation*}
Thus, by Lemma \ref{lem:self-similar} applied with $k+1$ in place of $k$,
\begin{align*}
T(x)&=\frac{k+1}{2^{k+1}}+\frac14 T\left(4\left(x-\frac{1}{2^{k+1}}\right)\right)\\
&<\frac{k+1}{2^{k+1}}+\frac14\cdot \frac{k-1}{2^{k-1}}=\frac{k}{2^k}=T\left(\frac{1}{2^k}\right),
\end{align*}
completing the proof.
\end{proof}

\begin{proof}[Proof of Theorem \ref{thm:fundamental-equation}]
Statement (i) is obvious. To prove statement (ii), fix $y\in I_k$ with $k\geq 3$. We can divide $L_0(y)$ in three parts, namely its intersections with the intervals $[0,(\frac{1}{2})^k]$, $[(\frac{1}{2})^k,(\frac{1}{2})^{k-1}]$ and $[(\frac{1}{2})^{k-1},\frac12]$. By Lemma \ref{lem:self-similar},
\begin{equation*}
L_0(y)\cap \left[(\tfrac{1}{2})^k,(\tfrac{1}{2})^{k-1}\right] =f_k\left[L_0(\Psi(y))\right],
\end{equation*}
since, for $T(x)\in I_k$, \eqref{eq:local-similarity} can be written as $\Psi(T(x))=T\big(f_k^{-1}(x)\big)$.
By Lemma \ref{lem:left-side},
\begin{equation*}
L_0(y)\cap \left[0,(\tfrac{1}{2})^k\right]=\bigcup_{j=0}^\infty g_{k,j}\left[L(4^j\Phi(y))\right],
\end{equation*}
since \eqref{eq:small-humps-to-the-left} can be written as $4^j\Phi(T(x))=T\big(g_{k,j}^{-1}(x)\big)$. Finally,
\begin{equation*}
L_0(y)\cap \left[(\tfrac{1}{2})^{k-1},\tfrac12\right]=\emptyset
\end{equation*}
in view of Lemma \ref{lem:pseudo-monotone}, applied with $k-1$ in place of $k$. Thus, we have \eqref{eq:set-equation2}. It is easy to check that all parts of the union are disjoint provided $y\neq t_k$.

Statement (iii) follows similarly from Lemmas \ref{lem:self-similar} and \ref{lem:left-side} (take $k=2$) by considering the intersection of $L_0(y)$ with $[0,\frac14]$ and $[\frac14,\frac12]$, respectively. The parts of the union are disjoint as long as $y\neq \frac12$.
\end{proof}

\section{Level sets with exactly two elements} \label{sec:two-elements}

In this section we focus on the set
\begin{equation*}
S_2:=\{y\in[0,\tfrac23]: |L(y)|=2\}.
\end{equation*}
We establish conditions for membership in this set and obtain bounds on its Lebesgue measure.

First, define a function $\kappa:[0,\frac23]\to\{2,3,\dots,\infty\}$ by
\begin{equation*}
\kappa(y)=\begin{cases}
\mbox{the number $k$ such that $y\in I_k$}, & \mbox{if $0<y\leq\frac23$},\\
\infty, & \mbox{if $y=0$},
\end{cases}
\end{equation*}
and let
\begin{equation*}
\kappa_n(y):=\kappa(\Psi^n(y)), \qquad n=0,1,\dots, \quad y\in[0,\tfrac23].
\end{equation*}

It is plain from the graph of $T$ that $|L(y)|\geq 4$ for $\frac12\leq y\leq \frac23$. It is also clear that $|L(0)|=2$. Thus, we need only consider points $y$ with $0<y<\frac12$. Note that for points in this interval, $\kappa_n(y)\geq 3$ for each $n$.

\begin{theorem} \label{thm:when-just-two}
Let $0<y<\frac12$. Then $|L(y)|=2$ if and only if
\begin{equation}
\Phi(\Psi^n(y))>\frac23 \qquad\mbox{for all $n\geq 0$}.
\label{eq:two-condition}
\end{equation}
\end{theorem}

\begin{proof}
The theorem is an easy consequence of Theorem \ref{thm:fundamental-equation}. Let $y_n:=\Psi^n(y)$, and $k_n:=\kappa_n(y)=\kappa(y_n)$, for $n=0,1,2,\dots$. Suppose that \eqref{eq:two-condition} holds. Then $4^j\Phi(y_n)>\frac23$ for all $j\geq 0$, so  \eqref{eq:set-equation2} gives
$L_0(y_n)=f_{k_n}(L_0(y_{n+1}))$
for each $n$. But then
\begin{equation*}
L_0(y)=f_{k_0}\circ f_{k_1}\circ\dots\circ f_{k_{n-1}}(L_0(y_n)),
\end{equation*}
for each $n$. This implies
\begin{equation*}
\diam L_0(y)=\diam\left(f_{k_0}\circ f_{k_1}\circ\dots\circ f_{k_{n-1}}(L_0(y_n))\right)
\leq\frac12\left(\frac14\right)^n \to 0.
\end{equation*}
Hence $|L_0(y)|=1$, and so $|L(y)|=2$. This proves the ``if" part.

Conversely, if there is an $n$ such that $\Phi(y_n)\leq \frac23$,
then $L(\Phi(y_n))\neq\emptyset$. But then $|L_0(y_n)|\geq 2$ by \eqref{eq:card-equation2}, so that
\begin{equation*}
|L_0(y)|\geq|f_{k_0}\circ f_{k_1}\circ\dots\circ f_{k_{n-1}}(L_0(y_n))|\geq 2.
\end{equation*}
Thus $|L(y)|\geq 4$, proving the ``only if" part.
\end{proof}

While the condition in Theorem \ref{thm:when-just-two} is exact, it is in general difficult to verify. The following corollary gives a useful and easy-to-check sufficient condition in terms of the binary expansion of $y$.

\begin{corollary} \label{cor:three-is-too-many}
Let $0<y<\frac12$ such that $y$ is not a dyadic rational, and suppose the binary expansion of $y$ does not contain a string of three consecutive $0$'s anywhere after the occurrence of its first $1$. More precisely, write $y=\sum_{n=1}^\infty 2^{-n}\omega_n$ with $\omega_n\in\{0,1\}$, and suppose there do not exist indices $k$ and $l$ with $k<l$ such that $\omega_k=1$, and $\omega_l=\omega_{l+1}=\omega_{l+2}=0$. Then $|L(y)|=2$.
\end{corollary}

\begin{proof}
Define $y_n$ and $k_n$ as in the proof of Theorem \ref{thm:when-just-two}.
We claim that for each $n\geq 0$, the binary expansion of $y_n$ does not have three consecutive zeros anywhere past its $k_n$-th digit. This is obvious for $n=0$. Suppose it holds for some $n\geq 0$. Then, since
$$y_{n+1}=4\left(y_n-\frac{k_n}{2^{k_n}}\right),$$
the binary expansion of $y_{n+1}$ will not have three consecutive zeros anywhere past its $(k_n-2)$-nd digit. Therefore, since $k_{n+1}\geq k_n-1$, the binary expansion of $y_{n+1}$ certainly does not have three consecutive zeros anywhere past its $k_{n+1}$-st digit, proving the claim.

Since $k_n\geq 3$, it now follows that for each $n$,
\begin{equation*}
\Phi(y_n)=4^{k_n}\left(y_n-\frac{k_n}{2^{k_n}}\right)\geq 4^{k_n}\left(\frac12\right)^{k_n+3}=2^{k_n-3}>\frac23.
\end{equation*}
Hence, by Theorem \ref{thm:when-just-two}, $|L(y)|=2$.
\end{proof}

Thus, for instance, the level sets at levels $\frac13$, $\frac15$, $\frac25$, $\frac16$, $\frac17$, $\frac27$, $\frac37$ all have precisely two elements. It is clear from Corollary \ref{cor:three-is-too-many} that there are uncountably many ordinates $y$ having this property. In fact, there exist uncountably many such ordinates in each interval $I_k$, where $k\geq 3$. But the corollary does not imply that the set $S_2$ has positive Lebesgue measure. This stronger statement will follow, however, from Theorem \ref{thm:measure-of-two} below.

The following corollary gives a slightly weaker sufficient condition and an accompanying necessary condition, which together nearly characterize which $y$ have $|L(y)|=2$ in terms of the sequence $\{k_n\}$.

\begin{corollary} \label{cor:double-at-most}
Let $y\in(0,\frac12)$, and let $k_n:=\kappa_n(y)$. If 
\begin{equation*}
k_{n+1}\leq 2k_n+\log_2 k_n+\log_2 3-2
\end{equation*}
for each $n$, then $|L(y)|=2$. In particular, $|L(y)|=2$ if the sequence $\{k_n\}$ at most doubles at each step; that is, if $k_{n+1}\leq 2k_n$ for each $n$.

On the other hand, if
\begin{equation*}
k_{n+1}\geq 2k_n+\log_2 k_n+\log_2 3
\end{equation*}
for some $n$, then $|L(y)|>2$.
\end{corollary}

\begin{proof}
Let $y_n=\Psi^n(y)$, and suppose that for some $n$, $\Phi(y_n)\leq\frac23$; that is,
\begin{equation}
y_n-\frac{k_n}{2^{k_n}}\leq \frac23\left(\frac14\right)^{k_n}.
\label{eq:bad-interval}
\end{equation}
Then
\begin{equation*}
\frac{3}{2^{k_{n+1}}}\leq \frac{k_{n+1}}{2^{k_{n+1}}}\leq y_{n+1}=4\left(y_n-\frac{k_n}{2^{k_n}}\right)\leq \frac23\left(\frac14\right)^{k_n-1},
\end{equation*}
from which it follows that $2^{k_{n+1}}\geq(9/2)4^{k_n-1}>4^{k_n}$, and hence $k_{n+1}>2k_n$. Putting this back into the lower estimate above gives
\begin{equation*}
\frac{2k_n}{2^{k_{n+1}}}<\frac{k_{n+1}}{2^{k_{n+1}}}\leq \frac23\left(\frac14\right)^{k_n-1},
\end{equation*}
so that $2^{k_{n+1}}>3k_n 4^{k_n-1}$. Taking logarithms, we obtain $k_{n+1}>2k_n+\log_2 k_n+\log_2 3-2$.

For the second statement, we use the fact that
\begin{equation}
\mbox{if}\quad k\geq\log_2 u+\log_2\log_2 u+1\quad \mbox{for}\quad u\geq 4,\quad  \mbox{then}\quad \frac{k}{2^k}\leq \frac{1}{u}. 
\label{eq:log-condition}
\end{equation}
This follows since $k/2^k$ is nonincreasing, and $\log_2\log_2 u+1\leq\log_2 u$ when $u\geq 4$.

If \eqref{eq:bad-interval} fails for some $n$, then the definition of $k_{n+1}$ gives
\begin{equation*}
\frac{k_{n+1}-1}{2^{k_{n+1}-1}}>y_{n+1}=4\left(y_n-\frac{k_n}{2^{k_n}}\right)>\frac23\left(\frac14\right)^{k_n-1},
\end{equation*}
so applying \eqref{eq:log-condition} with $u=\frac32\cdot 4^{k_n-1}$, it follows that
\begin{align*}
k_{n+1}-1&<\log_2\left(\tfrac32\cdot 4^{k_n-1}\right)+\log_2 \log_2\left(\tfrac32\cdot 4^{k_n-1}\right)+1\\
&=2k_n+\log_2 3-2+\log_2(2k_n+\log_2 3-3)\\
&<2k_n+\log_2 k_n+\log_2 3-1.
\end{align*}
Hence, $k_{n+1}<2k_n+\log_2 k_n+\log_2 3$.
\end{proof}

Corollary \ref{cor:double-at-most} implies, for example, that $|L(y)|=2$ whenever $y$ is the fixed point of a composition $\Psi_{k_n}\circ\Psi_{k_{n-1}}\circ\dots\circ\Psi_{k_1}$ with $k_{j+1}\leq 2k_j$ for $j=1,\dots,n-1$, and $k_1\leq 2k_n$. This leads to many more examples. In particular, the fixed point of each $\Psi_k$ with $k\geq 4$ has this property. (Only $\Psi_3$ does not have a fixed point in $[0,\frac12)$.) It is easy to calculate that the fixed point of $\Psi_k$ is
\begin{equation*}
y_k^*:=\frac{4t_k}{3}=\frac{k}{3\cdot 2^{k-2}}, \qquad k\geq 4.
\end{equation*}
Note that, surprisingly perhaps, every third number in this sequence is a dyadic rational. For instance, $y_6^*=1/8$, $y_9^*=3/2^7$, $y_{12}^*=1/2^8$, etc.

\begin{example}
{\rm
The binary expansion of $1/11$ is $0.\overline{0001011101}$, which does not satisfy the ``no 3 zeros" condition of Corollary \ref{cor:three-is-too-many}. But $1/11$ is the fixed point of the ten-fold composition $\Psi_4^3\circ\Psi_5^2\circ\Psi_6\circ\Psi_5^2\circ\Psi_6\circ\Psi_7$. Thus, by Corollary \ref{cor:double-at-most}, $|L(1/11)|=2$.
}
\end{example}

An intriguing question, which is a variant of one raised by Knuth \cite[Exercise 83]{Knuth}, is: given a rational $y$, can one always determine in a finite number of steps whether $|L(y)|=2$? If the sequence $\{(k_n,y_n)\}$ is eventually periodic, then one has to check the condition \eqref{eq:two-condition} for only finitely many $n$. But there are in fact many rational numbers $y$ for which $\{k_n\}$ never repeats: take, for example, $y=T(1/7)=22/49$, which has $k_n=n+3$ for every $n$. For this $y$, Corollary \ref{cor:double-at-most} nonetheless gives $|L(y)|=2$.

\subsection{The measure of $S_2$} \label{subsec:measure}

\begin{theorem} \label{thm:measure-of-two}
The set $S_2$ is nowhere dense and $G_\delta$. It is not closed. Its Lebesgue measure $\lambda(S_2)$ satisfies
\begin{equation}
\frac{5}{12}<\lambda(S_2)<\frac{35}{72}.
\label{eq:S2-sandwich}
\end{equation}
\end{theorem}

To prove the theorem, we need to count the first-generation humps of order $m$. This involves the 
{\em Catalan numbers}
\begin{equation*}
C_n:=\frac{1}{n+1}\binom{2n}{n}, \qquad n=0,1,2,\dots,
\end{equation*}
which satisfy the identity
\begin{equation}
\sum_{n=0}^\infty C_n\left(\frac14\right)^n=2.
\label{eq:Catalan-gf}
\end{equation}

\begin{lemma} \label{lem:hump-count} 
For each $m\in\NN$, the graph $\mathcal{G}_T$ contains precisely $C_{m-1}$ first-generation leading humps of order $m$.
\end{lemma}

\begin{proof}
Each hump of order $m$ corresponds uniquely to a path of $m$ steps starting at $(0,0)$, taking steps $(1,1)$ or $(1,-1)$, and ending at $(2m,0)$. It is well known that exactly $C_m$ of these paths stay on or above the horizontal axis (see Feller \cite[p.~73]{Feller}). Now each first-generation leading hump of order $m$ corresponds to a path with first step $(1,1)$ and last step $(1,-1)$, and which stays strictly above the horizontal axis in between these two steps. By translation, this is the same as the number of paths from $(0,0)$ to $(2m-2,0)$ which do not go below the horizontal axis; this number is therefore $C_{m-1}$.
\end{proof}

\begin{proof}[Proof of Theorem \ref{thm:measure-of-two}]
Recalling Definition \ref{def:balanced}, let $\BB$ denote the set of all balanced dyadic rationals in $[0,1)$.
Observe that $S_2$ is obtained from $[0,\frac23]$ by removing the projections onto the $y$-axis of all first-generation humps, of which there are countably many. (Recall that these projections are intervals of the form $J(x_0)$, where $x_0\in\BB$.) Hence, $S_2$ is $G_\delta$. It is not closed, because, for example, the point $\frac12$ does not lie in $S_2$ but can be approximated from below by points in $S_2$ (take $x=0.01^m (01)^\infty$, for instance, which is in $S_2$ by Corollary \ref{cor:three-is-too-many}, and let $m\to\infty$). That $S_2$ is nowhere dense is shown in \cite[Theorem 4.2]{Allaart}.

To estimate the measure of $S_2$, we show that
\begin{equation}
\frac{13}{72}<\lambda\left(\bigcup_{x_0\in\BB}J(x_0)\right)<\frac{1}{4}.
\label{eq:sandwich}
\end{equation}
For the lower bound, note that by Theorem \ref{thm:when-just-two}, the collection $\{J(x_0): x_0\in\BB\}$ contains the disjoint family of intervals $\{J_1,J_3,J_4,\dots\}$, where $J_k:=[t_k,t_k+\frac23(\frac14)^k]$. (The interval $J_2$ is contained in $J_1$, which is just $[\frac12,\frac23]$.)  Thus,
\begin{equation*}
\lambda\left(\bigcup_{x_0\in\BB}J(x_0)\right)\geq \diam(J_1)+\sum_{k=3}^\infty \diam(J_k)=\frac16+\sum_{k=3}^\infty \frac23\left(\frac14\right)^k=\frac{13}{72}.
\end{equation*}
Since $S_2$ is nowhere dense, there are intervals $J(x_0)$ which are not contained in $\bigcup_{k=1}^\infty J_k$, so we have in fact strict inequality in the first half of \eqref{eq:sandwich}.

The upper bound uses a simple counting argument. For each $m\in\NN$ there are $C_{m-1}$ first-generation leading humps by Lemma \ref{lem:hump-count}. However, by Lemma \ref{lem:series-of-humps} each first-generation leading hump $H$ of order $m$ has directly to its left an infinite sequence of smaller first-generation leading humps, of orders $m+1,m+2,\dots$, which we call {\em subsidiary humps}. We need not count these, since their projections onto the $y$-axis are contained in that of $H$. Consequently, a first-generation leading hump of order $m$ should not be counted if it is a subsidiary hump to a first-generation leading hump of order $m-1$. Of these, there are exactly $C_{m-2}$. Setting $C_{-1}:=0$, we thus obtain the upper estimate
\begin{align}
\begin{split}
\lambda\left(\bigcup_{x_0\in\BB}J(x_0)\right)&\leq \sum_{m=1}^\infty(C_{m-1}-C_{m-2})\cdot \frac23\left(\frac14\right)^m\\ &=\left(\frac23\cdot\frac14-\frac23\cdot\frac{1}{4^2}\right)\sum_{n=0}^\infty C_n\left(\frac14\right)^n=\frac14,
\end{split}
\label{eq:measure-upper-bound}
\end{align}
where the last equality uses \eqref{eq:Catalan-gf}.
Here too we have in fact strict inequality, as some of the intervals $J(x_0)$ counted in \eqref{eq:measure-upper-bound} overlap each other.

Since $\lambda(S_2)=\frac23-\lambda\left(\bigcup_{x_0\in\BB}J(x_0)\right)$, the estimate \eqref{eq:S2-sandwich} follows.
\end{proof}

\begin{remark}
{\rm
The bounds for $\lambda(S_2)$ on both sides can be somewhat improved by examining more closely the degree of overlap between the first-generation removed intervals. However, the calculations become quite cumbersome, and it seems difficult to significantly narrow the interval of \eqref{eq:S2-sandwich}.
}
\end{remark}

The result of Theorem \ref{thm:measure-of-two} should not be suprising when one observes the graph of the Takagi function. The result of Buczolich \cite{Buczolich} says that almost all level sets are finite, and it is certainly plausible that $2$ is the most common cardinality.

\subsection{Takagi expansions and solutions of $T(x)=y$}
\label{subsec:takagi-expansion}

For nondifferentiable functions, finding even approximate solutions to the equation $T(x)=y$ is a nontrivial task, since there is no obvious replacement for Newton's method. Here we show, as a by-product of our analysis, how the sequence $\{k_n\}$ can be used to solve this problem for the Takagi function.

\begin{definition}
{\rm
For a point $y\in[0,\frac23]$, we call the sequence $\{k_n\}$ defined by $k_n=\kappa_n(y)$ the (canonical) {\em Takagi expansion} of $y$, and write $y=[k_0,k_1,\dots]$. If $k_i=k$ for all $i\geq n$, we write $y=[k_0,\dots,k_{n-1},\bar{k}]$. Instead of the expansion $[k_0,\dots,k_n,\overline{\infty}]$ we write simply $[k_0,\dots,k_n]$.
}
\end{definition}

\begin{example}
{\rm
We have $1/2=[2]$, $1/3=[\bar{4}]$, $2/3=[\bar{2}]$, $3/8=[3]$, $19/32=[2,3]$, $3/7=[3,5,5,4,5,5,4,\dots$].
}
\end{example}

The Takagi expansion of a point $y$ can be used to approximate a solution to the equation $T(x)=y$. From the definition of $k_n$ we see that
\begin{equation}
y=\sum_{n=0}^\infty \frac{k_n}{2^{k_n}4^n}=\sum_{n=0}^\infty \frac{k_n}{2^{k_n+2n}},
\label{eq:y-representation}
\end{equation}
where we interpret the $n$-th term of the series as $0$ when $k_n=\infty$. Put
\begin{equation}
x=\sum_{n=1}^\infty 2^{-l_n}, \qquad l_n:=k_{n-1}+2(n-1),\ \ n\in\NN.
\label{eq:x-construction}
\end{equation}
Then $T(x)=y$, as can be seen easily using Lemmas \ref{lem:self-similar} and \ref{lem:pseudo-monotone}, induction, and the continuity of $T$. 

For the canonical Takagi expansion, we have $k_n\geq 2$, $k_{n+1}\geq k_n-1$, and if $k_n\geq 3$, then $k_{n+1}\geq 3$. With these requirements, the representation \eqref{eq:y-representation} is unique. However, we can obtain more solutions of $T(x)=y$ in $[0,\frac12]$ (provided they exist) by relaxing the conditions on the sequence $\{k_n\}$. Specifically, we can drop the last requirement and demand merely that $k_n\geq 2$ and $k_{n+1}\geq k_n-1$ for all $n$. This can yield alternative representations of the form \eqref{eq:y-representation}, which we also call Takagi expansions and which correspond to different solutions of $T(x)=y$. The idea is based on the identity
\begin{equation}
\frac{k}{2^k}=\sum_{j=2}^{k+1}\frac{j}{2^j 4^{k-j+1}},
\end{equation}
which implies that $[k_0,\dots,k_{n-1},k_n]=[k_0,\dots,k_{n-1},k_n+1,k_n,k_n-1,\dots,2]$. For instance, $y=3/8$ has the representations $[3]$, $[4,3,2]$, $[4,3,3,2]$, etc., corresponding to the solutions $x=1/8$, $x=7/64$ and $x=27/256$, etc. Analogously, $y=5/32=[5]=[6,5,4,3,2]$. Starting with the canonical Takagi expansion of $y$, one can determine whether there exist additional representations as follows. If $4^{k_n}(y_n-t_{k_n})>\frac23$ for all $n$, then the Takagi expansion is unique. On the other hand, if for some $n$, $4^{k_n}(y_n-t_{k_n})\leq\frac23$, then $y$ has an alternative Takagi expansion
\begin{equation}
y=[k_0,\dots,k_{n-1},k_n+1,k_n,k_n-1,\dots,2,k_{n+k_n}',k_{n+k_n+1}',\dots].
\end{equation}
To find it, put $y'=4^{k_n}(y_n-t_{k_n})$, and let $k_{n+k_n+j}'=\kappa_j(y')$ for $j=0,1,\dots$. This procedure can be repeated for any Takagi expansion of $y$ and at any position $n$ such that $4^{k_n}(y_n-t_{k_n})\leq\frac23$. As an example, the point $y=377/2048$ has canonical Takagi expansion $[3,9]$, with corresponding solution $x=257/2048$. Since $4^3(y-t_3)=9/32<2/3$, and $9/32=[4,\bar{6}]$, $y$ has the additional Takagi expansion $[4,3,2,4,\bar{6}]$, with corresponding solution $x=1357/12288$. In general, a given point $y$ may have finitely many, countably many or uncountably many Takagi expansions. 

The solutions of $T(x)=y$ corresponding to different Takagi expansions of $y$ are not only different, but represent different local level sets as defined by Lagarias and Maddock \cite{LagMad1}. Define an equivalence relation $\sim$ on $[0,1]$ by saying that $x\sim x'$ if $|D_n(x)|=|D_n(x')|$ for all $n$. The {\em local level set} determined by $x$ is the set $L_x^{loc}:=\{x'\in[0,1]: x'\sim x\}$. Points inside a local level set are easily obtained from one another by simple operations (``block flips") on their binary expansions -- see \cite{LagMad1}.
The size of a local level set in $L(y)$ can be inferred from the number of 2's in the corresponding Takagi expansion of $y$: If the number 2 occurs exactly $m$ times in the sequence $\{k_n\}$, then $L_x^{loc}$ with $x$ defined by \eqref{eq:x-construction} has exactly $2^{m+1}$ elements (provided that we ``split" each dyadic rational point $x$ in two separate points $x_+$ and $x_-$, corresponding to the two possible binary representations of $x$). If it occurs infinitely often, $L_x^{loc}$ is uncountable. Moreover, the point $x$ obtained via \eqref{eq:x-construction} is always the leftmost point of $L_x^{loc}$, as one checks easily that $D_n(x)\geq 0$ for all $n$. To summarize:
\begin{itemize}
\item The number of local level sets contained in $L(y)$ equals the number of distinct Takagi expansions of $y$;
\item The leftmost point of the local level set associated with Takagi expansion $y=[k_0,k_1,\dots]$ is the point $x$ defined by \eqref{eq:x-construction};
\item The cardinality of the local level set is determined by the number of 2's in the associated Takagi expansion of $y$.
\end{itemize}

\section{General finite cardinalities} \label{sec:even-cardinalities}

The previous section was concerned mainly with level sets consisting of exactly two points. It is natural to ask which other cardinalities are possible, and whether they occur with positive probability. Of course, the cardinality of any finite level set must be even, in view of the symmetry of the graph of $T$. The next theorem shows that conversely, every even positive integer is the cardinality of some level set of $T$.

\begin{theorem} \label{thm:all-evens}
For every positive integer $n$, there exist uncountably many ordinates $y$ such that $|L(y)|=2n$.
\end{theorem}

\begin{proof}
By Corollary \ref{cor:three-is-too-many} (or Theorem \ref{thm:measure-of-two}) the statement is true for $n=1$. We show here that, for each $m\in\NN$, there are uncountably many level sets with cardinality $4m$, and uncountably many with cardinality $4m+2$. 

Let $\hat{y}$ be a point in $(\frac16,\frac12)$ satisfying the condition of Corollary \ref{cor:three-is-too-many}; note that there are uncountably many such points. Let $m\in\NN$, and put
\begin{equation*}
y=\frac12+\left(\frac14\right)^m\hat{y}, \qquad\mbox{and}\qquad
y'=\frac38+\left(\frac14\right)^{m+2}\hat{y}.
\end{equation*}
We first show that $|L(y)|=4m$. Observe that for $j\geq m$, $4^j\Phi(y)=4^{j-m}\hat{y}\geq 4\hat{y}>\frac23$, while for $j<m$, $4^j\Phi(y)\leq \hat{y}<\frac12$. Since $4^j\Phi(y)$ also satisfies the ``no 3 zeros" condition of Corollary \ref{cor:three-is-too-many}, it follows by \eqref{eq:card-equation1} and \eqref{eq:card-equation3} that $|L(y)|=4m$.

Next, we show that $|L(y')|=4m+2$. Note that $y'\in I_3$, so $\Phi(y')=4^3(y'-\frac38)$. Thus, we obtain again that $4^j\Phi(y')>\frac23$ for $j\geq m$, while $4^j\Phi(y')<\frac12$ for $j<m$. Since $4^j(y'-\frac38)$ satisfies the hypothesis of Corollary \ref{cor:three-is-too-many} for each $j\in\NN$, we conclude that
$|L_0(\Psi(y'))|=1$, and
\begin{equation*}
|L(4^j\Phi(y'))|=\begin{cases}
2, & \mbox{for $0\leq j<m$},\\
0, & \mbox{for $j\geq m$}.
\end{cases}
\end{equation*}
Hence, by \eqref{eq:card-equation2}, $|L_0(y')|=1+2m$, so that $|L(y')|=4m+2$.
\end{proof}

For $n\in\NN$, define the set
\begin{equation*}
S_{2n}:=\{y\in[0,\tfrac23]: |L(y)|=2n\}.
\end{equation*}
In \cite[Theorem 4.2]{Allaart}, we show that $S_{2n}$ is nowhere dense for each $n$, so these sets are small topologically speaking. On the other hand, we believe them to have positive Lebesgue measure.

\begin{conjecture}
For every positive integer $n$, $\lambda(S_{2n})>0$.
\end{conjecture}

It is natural to try to use the construction in the proof of Theorem \ref{thm:all-evens} as the basis for proving this conjecture, but this does not seem to work. In the last two theorems of this section, which verify the conjecture for the case when $2n$ is the sum or difference of two powers of $2$, we use a different approach which delves deeper into the hierarchical structure of humps.

Recall the definition of the intervals $I_k=[t_k,t_{k-1})$, $k\geq 3$.
It was observed earlier that for each $k\geq 3$, $\Psi$ maps $I_k$ onto $[0,t_{k-2})=\{0\}\cup\bigcup_{j=k-1}^\infty I_j$. Thus, for $0<y<\frac12$, the sequence $\{k_n\}$ from the proof of Theorem \ref{thm:when-just-two} satisfies 
\begin{equation}
k_0\geq 3 \quad\mbox{and}\quad k_n\geq\max\{3,k_{n-1}-1\} \quad\mbox{for $n\geq 1$},
\label{eq:k-condition}
\end{equation}
and every sequence $\{k_n\}$ satisfying these conditions is possible, and in fact determines a unique ordinate $y\in(0,\frac12)$ via \eqref{eq:y-representation}.

For each $k=3,4,\dots$, define a subinterval $J_k$ of $I_k$ by
\begin{equation*}
J_k:=\left[t_k,t_k+\tfrac23{(\tfrac14)}^k\right]. 
\end{equation*}
Let $\JJ_0$ denote the collection of intervals $J_k$, $k\geq 3$. For $n\geq 1$, let $\JJ_n$ be the collection of all intervals of the form
\begin{equation*}
\Psi_{k_0}^{-1}\circ\dots\circ\Psi_{k_{n-1}}^{-1}(J_{k_n})
\end{equation*}
such that the $(n+1)$-tuple $(k_0,\dots,k_n)$ satisfies \eqref{eq:k-condition}. Put $\JJ:=\bigcup_{n=0}^\infty \JJ_n$. Note that $\JJ$ is precisely the collection of projections of all first-generation humps onto the $y$-axis, except those of subsidiary humps.

We first establish individual estimates for the measure of the intersection of $S_2$ with each of the intervals $I_k$. Define
\begin{equation}
\sigma_k:=\sum_{J\in\JJ}\diam(J\cap I_k), \qquad k=3,4,\dots.
\label{eq:Sk-def}
\end{equation}
In view of Theorem \ref{thm:when-just-two}, we have
\begin{equation}
\lambda(I_k\backslash S_2)\leq\sigma_k,\qquad k\geq 3.
\label{eq:local-bound}
\end{equation}
The proof of Theorem \ref{thm:measure-of-two} implies that
\begin{equation}
\sum_{k=3}^\infty \sigma_k=\frac{1}{12},
\label{eq:sigma-total}
\end{equation}
because the summation in \eqref{eq:measure-upper-bound} includes the hump $H(\frac14)$, which sits above the line $y=\frac12$ and has height $\frac16$. Subtracting this from the total of $\frac14$ in \eqref{eq:measure-upper-bound} gives \eqref{eq:sigma-total}. We now calculate the individual $\sigma_k$'s.

\begin{lemma} \label{lem:individual-sums}
With $\sigma_k$ defined as in \eqref{eq:Sk-def}, we have:
\begin{equation}
\sigma_3=\frac{1}{32},\qquad \mbox{and}\qquad
\sigma_k=\frac{1}{2^{k+1}}-\frac{1}{2^{2k-1}}\quad \mbox{for $k\geq 4$}.
\label{eq:sigmas}
\end{equation}
\end{lemma}

\begin{proof}
Since $k_{n+1}\geq \max\{k_n-1,3\}$, we have the recursive relations
\begin{align}
\sigma_3&=\frac23\left(\frac14\right)^3+\frac14\sum_{j=3}^\infty \sigma_j, \notag\\
\sigma_k&=\frac23\left(\frac14\right)^k+\frac14\sum_{j=k-1}^\infty \sigma_j, \qquad k\geq 4.
\label{eq:Sk}
\end{align}
Using \eqref{eq:sigma-total} it follows immediately that
\begin{equation}
\sigma_3=\frac{1}{32}, \qquad\mbox{and}\qquad \sigma_4=\frac{3}{128}=\frac{1}{2^{4+1}}-\frac{1}{2^{2\cdot 4-1}}.
\label{eq:3and4}
\end{equation}
Next, \eqref{eq:Sk} gives the difference equation
\begin{equation*}
\sigma_{k+1}=\sigma_k-\frac{\sigma_{k-1}}{4} -\frac12\left(\frac14\right)^k, \qquad k\geq 4.
\end{equation*}
From this equation and the initial values \eqref{eq:3and4}, \eqref{eq:sigmas} follows easily by induction.
\end{proof}

\begin{proposition} \label{prop:positive-measure}
(i) For each $k\geq 3$, $\lambda(S_2\cap I_k)>0$.

(ii) For each $m\in\NN$, $\lambda(S_{2^m})>0$.

(iii) If $\lambda(S_n\cap I_3)>0$ for some $n\in \NN$, then $\lambda(S_{2^m n})>0$ for every $m\in\NN$.
\end{proposition}

\begin{proof}
Lemma \ref{lem:individual-sums} and \eqref{eq:local-bound} together imply that $\lambda(S_2\cap I_k)\geq t_{k-1}-t_k-\sigma_k>0$, proving (i). In particular, $\lambda(S_2\cap I_3)>0$. Suppose $y\in S_2\cap I_3$, and put
\begin{equation*}
y_m:=\sum_{i=0}^{m-2}\frac12\left(\frac14\right)^i+\frac{y}{4^{m-1}}.
\end{equation*}
Then $y_m\in S_{2^m}$. This is trivial if $m=1$; assume it holds for some $m\in
\NN$. Since $y_{m+1}\geq\frac12$, we have $\Phi(y_{m+1})=4(y_{m+1}-\frac12)=y_m$, and since $y_m>\frac16$, \eqref{eq:card-equation3} gives $|L(y_{m+1})|=2|L(y_m)|=2^{m+1}$. Thus, $y_{m+1}\in S_{2^{m+1}}$, completing the induction. It now follows that $\lambda(S_{2^m})\geq{(\frac14)}^{m-1}\lambda(S_2\cap I_3)>0$. This proves (ii); statement (iii) follows similarly.
\end{proof}

In fact, it is easy to check that $\lambda(S_2\cap I_k)/\lambda(I_k)\to 1$ as $k\to\infty$, so $S_2$ becomes more dense (in the sense of probability) as one gets closer to the bottom of the graph.

The proof of Proposition \ref{prop:positive-measure} illustrates a typical use of Corollary \ref{cor:cardinality}. A variant of the argument is the following:
if $0<y<\frac12$ and $\Phi(y)>\frac16$, then \eqref{eq:card-equation2} reduces to $|L_0(y)|=|L_0(\Psi(y))|+|L(\Phi(y))|$, or equivalently,
\begin{equation}
|L(y)|=|L(\Psi(y))|+2|L(\Phi(y))|.
\label{eq:simpler-card-equation}
\end{equation}
If in fact $0<y<\frac12$ and $\Phi(y)>\frac23$, we obtain even more simply that $|L(y)|=|L(\Psi(y))|$. We will use these results several times in the proofs below.

\begin{theorem} \label{thm:2sum}
Let $n\in\NN$. If there are distinct integers $k$ and $l$ such that $2n=2^k+2^l$, then $\lambda(S_{2n})>0$.
\end{theorem}

\begin{proof}
By Proposition \ref{prop:positive-measure}(iii) it is enough to show that  $\lambda(S_{2^m+2}\cap I_3)>0$ for all $m\geq 2$. We will show that the interval
\begin{equation*}
U_m:=\left(\frac38+\sum_{j=1}^{m-2}\frac{1}{2^{2j+5}}+\frac{3}{2^{2m+5}},\frac38+\sum_{j=1}^{m-1}\frac{1}{2^{2j+5}}\right),
\end{equation*}
which clearly lies in $I_3$, contains a subset of $S_{2^m+2}$ of positive measure. 

Let $y\in U_m$. Then $\Phi(y)=4^3(y-\frac38)>\frac16$, so by \eqref{eq:simpler-card-equation} we see that $|L(y)|=2^m+2$ if $\Psi(y)\in S_2$ and $\Phi(y)\in S_{2^{m-1}}$. Define
\begin{equation}
y':=4^{m-2}\left(\Phi(y)-\sum_{j=1}^{m-2}\frac{1}{2^{2j-1}}\right).
\label{eq:y-prime}
\end{equation}
Then one easily checks that $y'\in I_3$, and since
\begin{equation*}
\Phi(y)=\sum_{j=1}^{m-2}\frac{1}{2^{2j-1}}+\frac{y'}{4^{m-2}},
\end{equation*}
it follows as in the proof of Proposition \ref{prop:positive-measure} that $\Phi(y)\in S_{2^{m-1}}$ if $y'\in S_2$. We next derive a condition on $y'$ that guarantees $\Psi(y)\in S_2$.

\bigskip
\noindent {\bf Claim:} We have $\kappa_n(y)=8$ for $n=1,\dots,m-2$, and $\kappa_{m-1}(y)=9$.

\bigskip
Assume first that $m\geq 3$. Since $y\in U_m$ and $\kappa(y)=3$, we have
\begin{equation*}
\frac{8}{2^8}=\frac{1}{2^5}<\sum_{j=1}^{m-2}\frac{1}{2^{2j+3}}+\frac{3}{2^{2m+3}}<\Psi(y)=4\left(y-\frac{3}{8}\right)<\sum_{j=1}^{m-1}\frac{1}{2^{2j+3}}<\frac{7}{2^7},
\end{equation*}
so $\kappa_1(y)=8$. Now it follows inductively that, for $n=2,\dots,m-2$,
\begin{align}
\begin{split}
\frac{8}{2^8}&<\sum_{j=1}^{m-n-1}\frac{1}{2^{2j+3}}+\frac{3}{2^{2(m-n)+5}}<\Psi^n(y)\\
&=4\left(\Psi^{n-1}(y)-\frac{8}{2^8}\right)
<\sum_{j=1}^{m-n} \frac{1}{2^{2j+3}}<\frac{7}{2^7},
\end{split}
\label{eq:general-sandwich}
\end{align}
and $\kappa_n(y)=8$. For $n=m-2$ this gives
\begin{equation}
\frac{1}{2^5}+\frac{3}{2^9}<\Psi^{m-2}(y)<\frac{1}{2^5}+\frac{1}{2^7},
\label{eq:simple-sandwich}
\end{equation}
and iterating once more we obtain
\begin{equation*}
\frac{9}{2^9}<\frac{3}{2^7}<\Psi^{m-1}(y)=4\left(\Psi^{m-2}(y)-\frac{8}{2^8}\right)<\frac{1}{2^5}=\frac{8}{2^8}.
\end{equation*}
Thus $\kappa_{m-1}(y)=9$, establishing the Claim for the case $m\geq 3$. If $m=2$, then $\Psi^{m-2}(y)=y$, so \eqref{eq:simple-sandwich} is just the statement $y\in U_m$. Thus we obtain in the same way as above that $\kappa_1(y)=9$.

\bigskip
From \eqref{eq:general-sandwich} it follows also that 
for $n=1,\dots,m-2$,
\begin{equation*}
\Phi(\Psi^n(y))>\frac23.
\end{equation*}
As a result, the Claim yields for $y\in U_m$ that $\Psi(y)\in S_2$ if and only if $\Psi^{m-1}(y)\in S_2\cap I_9$, or equivalently, $\Psi(y)\in \Psi_8^{-(m-2)}(S_2\cap I_9)$.

Now we have $\Phi(y)=4^2\Psi(y)$ since $y\in I_3$, so by \eqref{eq:y-prime},
\begin{equation*}
y'=4^m\Psi(y)-4^{m-2}\sum_{j=1}^{m-2}\frac{1}{2^{2j-1}}=:4^m\Psi(y)-a_m.
\end{equation*}
Since the affine map that takes $y$ to $y'$ maps $U_m$ onto $I_3$ and expands by a factor $4^{m+1}$, we obtain
\begin{align*}
4^{m+1}\lambda(S_{2^m+2}\cap I_3)&\geq \lambda\left\{y': y'\in S_2\cap I_3\ \mbox{and}\ y'\in 4^m\Psi_8^{-(m-2)}(S_2\cap I_9)-a_m\right\}\\
&\geq \lambda(S_2\cap I_3)-\lambda\left(4^m\Psi_8^{-(m-2)}(I_9\backslash S_2)-a_m\right)\\
&=\lambda(S_2\cap I_3)-4^m 4^{-(m-2)}\lambda(I_9\backslash S_2)\\
&\geq \frac18-\sigma_3-4^2\sigma_9>\frac18-\frac{1}{32}-\frac{4^2}{2^{10}}=\frac{5}{64},
\end{align*}
where the second-to-last inequality follows by \eqref{eq:local-bound}, and the last inequality by Lemma \ref{lem:individual-sums}.
Thus, $\lambda(S_{2^m+2}\cap I_3)>0$.
\end{proof}

The proof of the next result is rather more complicated, and appears to depend on a coincidence; see Claim 2 in the proof below.

\begin{theorem} \label{thm:2difference}
Let $n\in\NN$. If there are distinct integers $k$ and $l$ such that $2n=2^k-2^l$, then $\lambda(S_{2n})>0$.
\end{theorem}

\begin{proof}
By Proposition \ref{prop:positive-measure}, it is enough to show that $\lambda(S_{2^m-2}\cap I_3)>0$ for all $m\geq 3$. The case $m=3$ actually follows from Theorem \ref{thm:2sum}, since $2^3-2=2^2+2$. Assume therefore that $m\geq 4$. We will show that the interval 
\begin{equation*}
U_m:=\left(\frac38\sum_{j=0}^{m-2}\frac{1}{2^{6j}},\frac38\sum_{j=0}^{m-3}\frac{1}{2^{6j}}+\frac{1}{2^{6(m-2)+1}}\right),
\end{equation*}
which clearly lies in $I_3$, contains a subset of $S_{2^m-2}$ of positive measure.

\bigskip
\noindent {\bf Claim 1.} Let $y\in U_m$ be such that $\Psi\big(\Phi^{n}(y)\big)\in S_2$ for each $n=0,\dots,m-3$, and $\Phi^{m-2}(y)\in S_2$. Then $y\in S_{2^m-2}$.

\bigskip
This statement is a tautology if $m=2$. Proceeding by induction, suppose the claim is true for some $m\geq 2$, and let $y\in U_{m+1}$ be such that $\Psi\big(\Phi^{n}(y)\big)\in S_2$ for each $n=0,\dots,m-2$, and $\Phi^{m-1}(y)\in S_2$. Then $\Phi(y)=4^3(y-\frac38)\in U_m$, so the induction hypothesis applied to $\Phi(y)$ in place of $y$ gives $\Phi(y)\in S_{2^m-2}$. Finally, $\Phi(y)\in U_m$ implies $\Phi(y)>\frac16$. Since $\Psi(y)\in S_2$, it follows from \eqref{eq:simpler-card-equation} that $|L(y)|=2(2^m-2)+2=2^{m+1}-2$,
as required.

\bigskip
For the remainder of the proof, fix $y\in U_m$ and define $y_n:=\Phi^{n}(y)$ for $n=0,1,\dots,m-3$. It is easy to verify inductively that $y_n\in I_3$ for each $n$.

\bigskip
\noindent {\bf Claim 2.} Suppose $m\geq 5$. Then  $\Psi^7(y_n)=\Psi(y_{n+2})$ for $n=1,2,\dots,m-4$.

\bigskip
This is a bit tedious. We show first that for $n=1,2,\dots,m-4$,
\begin{equation}
\left(\kappa_1(y_n),\dots,\kappa_6(y_n)\right)=(9,9,9,8,7,8).
\label{eq:k-list}
\end{equation}
Since it is easy to check that for each $n$, the interval for $y_n$ corresponding to $y\in U_m$ is contained in the interval for $y_{n+1}$, it suffices to prove \eqref{eq:k-list} for $n=m-4$. Using the fact that each $y_n\in I_3$, one calculates
\begin{equation*}
y_{m-4}=\Phi^{m-5}(y)=4^{3(m-5)}y-\frac38\sum_{i=1}^{m-5}4^{3i}=2^{6(m-5)}y-\frac38\sum_{i=1}^{m-5}2^{6i}.
\end{equation*}
It follows after some elementary arithmetic that
\begin{equation}
\frac38\left(1+\frac{1}{2^6}+\frac{1}{2^{12}}+\frac{1}{2^{18}}\right)<y_{m-4}<\frac38\left(1+\frac{1}{2^6}+\frac{1}{2^{12}}\right)+\frac{1}{2^{19}},
\label{eq:constant-interval}
\end{equation}
which is independent of $m$.

Put $z:=y_{m-4}$, and let $z_j:=\Psi^j(z)$ for $j\geq 0$. We have $\kappa_0(z)=3$, and by \eqref{eq:constant-interval},
\begin{equation}
\frac38\left(\frac{1}{2^4}+\frac{1}{2^{10}}+\frac{1}{2^{16}}\right)<z_1=4\left(z_0-\frac38\right)<\frac38\left(\frac{1}{2^4}+\frac{1}{2^{10}}\right)+\frac{1}{2^{17}}.
\label{eq:another-sandwich}
\end{equation}
This implies $9/2^9<z_1<8/2^8$, so $\kappa_1(z)=9$. 
Note that the leading term in the left hand side of \eqref{eq:another-sandwich} is $3/2^7$, which is a fixed point of the mapping $\Psi_9(y)=4(y-t_9)$. We now continue, obtaining successively:
\begin{gather*}
\frac{9}{2^9}<3\left(\frac{1}{2^7}+\frac{1}{2^{11}}+\frac{1}{2^{17}}\right)<z_2=4\left(z_1-\frac{9}{2^9}\right)<3\left(\frac{1}{2^7}+\frac{1}{2^{11}}\right)+\frac{1}{2^{15}}<\frac{8}{2^8},\\
\kappa_2(z)=9,\\
\frac{9}{2^9}<3\left(\frac{1}{2^7}+\frac{1}{2^9}+\frac{1}{2^{15}}\right)<z_3=4\left(z_2-\frac{9}{2^9}\right)<3\left(\frac{1}{2^7}+\frac{1}{2^9}\right)+\frac{1}{2^{13}}<\frac{8}{2^8},\\
\kappa_3(z)=9,\\
\frac{8}{2^8}<3\left(\frac{1}{2^6}+\frac{1}{2^{13}}\right)<z_4=4\left(z_3-\frac{9}{2^9}\right)<\frac{3}{2^6}+\frac{1}{2^{11}}<\frac{7}{2^7},\\
\kappa_4(z)=8,\\
\frac{7}{2^7}<\frac{1}{2^4}+\frac{3}{2^{11}}<z_5=4\left(z_4-\frac{8}{2^8}\right)<\frac{1}{2^4}+\frac{1}{2^9}<\frac{6}{2^6},\\
\kappa_5(z)=7,\\
\frac{8}{2^8}<\frac{1}{2^5}+\frac{3}{2^9}<z_6=4\left(z_5-\frac{7}{2^7}\right)<\frac{1}{2^5}+\frac{1}{2^7}<\frac{7}{2^7},\\
\kappa_6(z)=8.
\end{gather*}
This establishes \eqref{eq:k-list}, which we now use to compute
\begin{align*}
\Psi^7(y_n)&=4^7 y_n-\left(4^7\cdot\frac{3}{2^3}+4^6\cdot\frac{9}{2^9}+4^5\cdot\frac{9}{2^9}+4^4\cdot\frac{9}{2^9}
+4^3\cdot\frac{8}{2^8}+4^2\cdot\frac{7}{2^7}+4\cdot\frac{8}{2^8}\right)\\
&=4^7 y_n-6240-\frac32.
\end{align*}
On the other hand,
\begin{equation*}
y_{n+2}=\Phi^2(y_n)=4^6 y_n-\frac38(4^6+4^3)=4^6 y_n-1560,
\end{equation*}
so that
\begin{equation*}
\Psi(y_{n+2})=4\left(y_{n+2}-\frac38\right)=4^7 y_n-6240-\frac32 =\Psi^7(y_n),
\end{equation*}
proving Claim 2.

\bigskip
\noindent {\bf Claim 3.} We have $\kappa_1(y_{m-2})=9$, and $\left(\kappa_1(y_{m-3}),\dots,\kappa_5(y_{m-3})\right)=(9,9,9,8,7)$.

\bigskip
Claim 3 is proved in the same way as \eqref{eq:k-list}, though the inequalities are slightly different.

\bigskip
\noindent {\bf Claim 4.} For $n=1,\dots,m-4$ and $j=1,\dots,6$, we have
$\Phi(\Psi^j(y_n))>\frac23$.

\bigskip
\noindent {\bf Claim 5.} For $j=1,\dots,5$, we have 
$\Phi(\Psi^j(y_{m-3}))>\frac23$.

\bigskip
Claims 4 and 5 are easy to check. Claim 4 follows from the sequence of inequalities beginning with \eqref{eq:another-sandwich} (recall that $z_j=\Psi^j(y_{m-4})$); Claim 5 is verified similarly.

\bigskip
It now follows from Claims 1,2 and 4 that $y\in S_{2^m-2}$ if and only if $\Psi(y_{m-3})\in S_2$, $\Psi(y_{m-2})\in S_2$, and $y_{m-1}\in S_2$. By Claims 3 and 5, this is the case if and only if
\begin{equation*}
y_{m-1}\in S_2\cap I_3, \qquad \Psi(y_{m-2})\in S_2\cap I_9, \qquad
\mbox{and}\quad \Psi^5(y_{m-3})\in S_2\cap I_7.
\end{equation*}
Since $y_{m-1}=4^2\Psi(y_{m-2})=4^5\Psi(y_{m-3})-24$, this last set of conditions holds if and only if $y_{m-1}$ lies in each of the sets
\begin{equation*}
S_2\cap I_3, \quad 4^2(S_2\cap I_9),
\quad\mbox{and}\quad 4^5\Psi_9^{-3}\circ\Psi_8^{-1}\left(S_2\cap I_7\right)-24,
\end{equation*}
in view of Claim 3.
Since the affine map that takes $y$ to $y_{m-1}$ maps $U_m$ onto $I_3$ and expands by a factor $4^{3(m-2)}$, we finally obtain
\begin{align*}
4^{3(m-2)}\lambda(S_{2^m-2}\cap I_3)&\geq \lambda(S_2\cap I_3) -\lambda\left(4^2(I_9\backslash S_2)\right)
-\lambda\left(4^5\Psi_9^{-3}\circ\Psi_8^{-1}(I_7\backslash S_2)-24\right)\\
&=\lambda(S_2\cap I_3)-4^2\lambda(I_9\backslash S_2)-4\lambda(I_7\backslash S_2)\\
&\geq \frac18-\sigma_3-4^2\sigma_9-4\sigma_7>\frac{1}{16},
\end{align*}
where the second-to-last inequality follows by \eqref{eq:local-bound}, and the last inequality by  Lemma \ref{lem:individual-sums}. Thus, $\lambda(S_{2^m-2}\cap I_3)>0$.
\end{proof}

\begin{remark}
{\rm
The ideas from the proofs of Theorems \ref{thm:2sum} and \ref{thm:2difference} can be combined to prove that $\lambda(S_{2n})>0$ for many more integers $n$. However, this method seems to break down in general when $n$ becomes too large.
}
\end{remark}

\end{document}